\documentclass[12pt, reqno]{amsart}

\usepackage{amsmath, amsthm, amscd, amsfonts, amssymb, graphicx, color}
\usepackage[bookmarksnumbered, colorlinks, plainpages]{hyperref}
\input{mathrsfs.sty}
\hypersetup{colorlinks=true,linkcolor=red, anchorcolor=green, citecolor=cyan, urlcolor=red, filecolor=magenta, pdftoolbar=true}

\newtheorem{theorem}{Theorem}[section]

\newtheorem{proposition}[theorem]{Proposition}
\newtheorem{corollary}[theorem]{Corollary}
\theoremstyle{definition}

\newtheorem{example}[theorem]{Example}

\theoremstyle{remark}
\newtheorem{remark}[theorem]{Remark}
\numberwithin{equation}{section}

\begin{document}

\title[Orthogonality preserving property for pairs of operators]
{Orthogonality preserving property for pairs of operators on Hilbert $C^*$-modules}

\author[M. Frank, M. S. Moslehian \MakeLowercase{and} A. Zamani]
{Michael Frank$^1$, Mohammad Sal Moslehian$^2$ \MakeLowercase{and} Ali Zamani$^3$}

\address{$^1$ Hochschule f\"ur Technik, Wirtschaft und Kultur (HTWK) Leipzig,
Fakult\"at IM, PF 301166, 04251 Leipzig, Germany}
\email{michael.frank@htwk-leipzig.de}

\address{$^2$ Department of Pure Mathematics, Ferdowsi University of Mashhad,
Center of Excellence in Analysis on Algebraic Structures (CEAAS),
P.O. Box 1159, Mashhad 91775, Iran}
\email{moslehian@um.ac.ir, moslehian@yahoo.com}

\address{$^3$ Department of Mathematics, Farhangian University, Tehran, Iran}
\email{zamani.ali85@yahoo.com}


\begin{abstract}
We investigate the orthogonality preserving property for pairs of operators on inner product
$C^*$-modules. Employing the fact that the $C^*$-valued inner product structure of a Hilbert
$C^*$-module is determined essentially by the module structure and by the orthogonality structure,
pairs of linear and local orthogonality-preserving operators are investigated, not a priori bounded.
We obtain that if $\mathscr{A}$ is a $C^{*}$-algebra
and $T, S:\mathscr{E}\to \mathscr{F}$ are two
bounded ${\mathscr A}$-linear operators between full Hilbert
$\mathscr{A}$-modules, then $\langle x, y\rangle = 0$ implies $\langle T(x), S(y)\rangle = 0$
for all $x, y\in \mathscr{E}$ if and only if there exists an element $\gamma$
of the center $Z(M({\mathscr A}))$ of the multiplier algebra $M({\mathscr A})$
of ${\mathscr A}$ such that $\langle T(x), S(y)\rangle = \gamma \langle x, y\rangle$
for all $x, y\in \mathscr{E}$. Varying the conditions on the operators $T$ and $S$ we obtain
further affirmative results for local operators and for pairs of a bounded and an unbounded
${\mathscr A}$-linear operator with bounded inverse.
\end{abstract} \maketitle
{\section{Introduction}
The starting point of considerations about orthogonality-preserving operators
on Hilbert spaces was Wigner's theorem \cite{W} with its first complete
proof by Uhlhorn \cite[Lemma 3.4, Theorems 4.1 and 4.2]{U}:
For two complex Hilbert spaces $(\mathscr{H},[\cdot,\cdot])$
and $(\mathscr{K},[\cdot,\cdot])$ with $\dim(\mathscr{H}) \geq 3$ and for a bijective operator
$T: \mathscr{H} \to \mathscr{K}$ with the property that $[T(x), T(y)] = 0$ if and only if $[x, y]=0$,
there exist a bijective isometry $U:\mathscr{H} \to \mathscr{K}$ and a scalar-valued function
$\phi:\mathscr{H} \to {\mathbb C}$} of modulus one such that $T(x)= \phi(x) U(x)$
for each $x \in \mathscr{H}$; see \cite{LT, MOL, RS} and the references therein.
Uhlhorn gave a counterexample for dimension $2$, and he found a similar statement
for two-dimensional Hilbert spaces under additional assumptions;
see \cite[Theorem 5.1]{U}. For a historical account on further variations
and extensions we refer the readers to the survey by Chevalier \cite{C}.
The situation of two unknown bijective operators $T,S: \mathscr{H} \to \mathscr{H}$ on
a given Hilbert space $\mathscr{H}$ of
dimension at least $3$ was treated by Moln\'{a}r \cite[Theorem 1]{Mol2000}:
In the case when $[T(x), S(y)] = [x, y]$ for each $x,y \in \mathscr{H}$
there are bounded invertible either both linear or both conjugate-linear operators
$U,V: \mathscr{H} \to \mathscr{H}$ such that $V={U^*}^{-1}$, $T=U$, and $S=V$.
Varying the conditions on $S$ and $T$, Chmieli\'{n}ski \cite{Ch.new} obtained a
number of exceptional vs. further affirmative results.

In \cite[Theorem 4]{Lu.Wo.1} {\L}ukasik and W\'ojcik were able to classify the Hilbert space situations in which $[T(x),S(y)]=[x,y]$ for
every $x, y\in \mathscr{H}$ and some functions $T,S: \mathscr{H}\to \mathscr{K}$: This takes place if and only if
there exist three suitable closed subspaces $\mathscr{M}_1, \mathscr{M}_2, \mathscr{M}_3 \subseteq \mathscr{K}$ such that
$\mathscr{K} = \mathscr{M}_1 + \mathscr{M}_2 + \mathscr{M}_3$ with $\mathscr{M}_j \bot \mathscr{M}_k$ for $j\not= k$,
and $T, S$ can be written as the following decompositions
\begin{align*}
T=A+ \phi, \qquad S=(A^*)^{-1} +\psi
\end{align*}
for an invertible operator $A: \mathscr{H}\to \mathscr{M}_1$ and for some operators
$\phi: \mathscr{H} \to \mathscr{M}_2$ and $\psi: \mathscr{H} \to \mathscr{M}_3$.
See \cite{Ch.new, Lu, Lu.Wo.2} for more results.

A first generalization of Wigner's theorem to Hilbert $C^*$-modules over standard $C^*$-algebras
was found by Ili{\v{s}}evi\'c and Turn{\v{s}}ek
\cite[Theorem 3.1]{I.T}. They also considered the case of approximate orthogonality-preservation.
In parallel, a generalization of Wigner's theorem to (full) Hilbert $\mathscr A$-modules $(\mathscr E, \langle \cdot,\cdot \rangle)$
was found by Frank et al. \cite[Theorem 4]{F.M.P} and by Leung et al.
\cite[Theorem 3.2, Corollary 3.3, and Theorem 3.4]{L.N.W.1} characterizing bounded/unbounded ${\mathscr A}$-linear operators $T: {\mathscr E} \to {\mathscr E}$ with the property that $\langle T(x),T(y) \rangle = 0$ whenever
$\langle x,y \rangle=0$ for $x,y \in {\mathscr E}$, by the equality $\langle T(x),T(y) \rangle = u \langle
x,y \rangle$ for a so-called $T$-specific positive central element $u$ of the multiplier algebra
of the $C^*$-algebra of coefficients and for all $x,y \in \mathscr E$.
The operator $T$ turns out to be bounded, and hence continuous.
Further, Leung et al. \cite{L.N.W.1} gave a characterization of such operators $T$ as
$T(x)=W(wx)=wW(x),\,\, (x \in \mathscr E)$, for a $T$-specific positive central element $w$ of the multiplier algebra
of the $C^*$-algebra of coefficients and for a (not necessarily bijective) Hilbert $C^*$-module isomorphism
$W: \overline{w{\mathscr E}} \to \overline{T({\mathscr E})}$. So, the orthogonality structure and
the $C^*$-module structure of a Hilbert $C^*$-module determine the Hilbert $C^*$-module,
without further topological characterizations beyond the definition.
Also, non-trivial orthogonality-preserving ${\mathscr A}$-linear operators on Hilbert $\mathscr A$-modules
with an injective $T$-specific positive central element $w$
(considered as a multiplication on the multiplier algebra) are always injective,
and hence, strongly orthogonality-preserving. Further conditions equivalent
to the orthogonality-preserving property of (bounded) ${\mathscr A}$-linear operators on Hilbert $\mathscr A$-modules
have been investigated by several authors, cf. \cite{A.R, A.R.2, I.T.Y, L.N.W.2, M.Z} and others.

In the present paper, we investigate several conditions on (not necessarily bijective)
${\mathscr A}$-linear operators $T$ and $S$ acting on Hilbert $\mathscr A$-modules and preserving
 the orthogonality of elements as a pair in one direction, that is, we do not require the bijectivity
of the operators, in general.
Also, we change orthogonality-preservation to approximate orthogonality preservation in some
situations, or we consider merely local operators. The exceptional cases for Hilbert space situations
indicate more complicated situations to appear for the more general Hilbert $C^*$-module settings.
We give some examples.
Our focus is on affirmative results of wide generality which can be obtained and on some proving
techniques to get more information on the background of the phenomena.
\section{Preliminaries}
Let $\big(\mathscr{H}, [\cdot, \cdot]\big)$ be an inner product space.
Recall that vectors $\eta, \zeta \in \mathscr{H}$ are said to be orthogonal,
written as $\eta\perp \zeta$, if $[\eta, \zeta] = 0$.
For inner product spaces $\mathscr{H}, \mathscr{K}$ and two functions
$T, S: \mathscr{H}\to \mathscr{K}$, the orthogonality preserving property
$$\eta\perp \zeta \Longrightarrow T(\eta)\perp S(\zeta) \qquad (\eta, \zeta\in \mathscr{H})$$
was introduced in \cite{Ch.new}. The following characterization was proved.
\begin{theorem}\cite[Theorem 3.9]{Ch.new}\label{th.1.1}
Let $\mathscr{H}$ and $\mathscr{K}$ be inner product spaces, and let $T, S: \mathscr{H}
\to \mathscr{K}$ be linear operators. The following conditions are equivalent:
\begin{itemize}
\item[(i)] $\eta\perp \zeta\, \Longrightarrow \,T(\eta)\perp S(\zeta)$ for all $\eta, \zeta \in\mathscr{H}$.
\item[(ii)] There exists $\gamma\in \mathbb{C}$ such that $[T(\eta), S(\zeta)]
= \gamma [\eta, \zeta]$ for all $\eta, \zeta \in\mathscr{H}$.
\end{itemize}
\end{theorem}
Some results in \cite{Ch.new} have been generalized in various ways by {\L}ukasik and W\'ojcik
in \cite{Lu.Wo.1, Lu.Wo.2}. Other related topics can be found in \cite{Lu, M.M.S}.

Notice that orthogonality preserving functions may be nonlinear and discontinuous,
i.e. far from linear; see \cite[Example 2]{J.C.4}. The theorem describes e.g. two situations:
Either the operators $S$ and $T$ may have orthogonal ranges and so $\gamma =0$
(and they may have non-trivial kernels), or in other situations ${\rm ker}(S)={\rm ker}(T)= \{ 0 \}$
for both kernels and ${\rm ran}(S)^{\bot\bot} ={\rm ran}(T)^{\bot\bot}$
for the ranges of both $S$ and $T$. We will see later that
in fact  ${\rm ran}(S)={\rm ran}(T)$, by Corollary \ref{Cor_2.10}.
However, for Hilbert $C^*$-modules over $C^*$-algebras of coefficients with non-trivial centers
the latter assumption may fail, cf. Example \ref{Ex_extra} below.

For a given $\theta\in[0, 1)$ two vectors $\eta, \zeta \in
\mathscr{H}$ are approximately orthogonal or $\theta$-orthogonal, denoted by
$\eta\perp^{\theta} \zeta$, if $\big|[\eta, \zeta]\big|\leq \theta\|\eta\|\,\|\zeta\|$.
Two operators $S, T: \mathscr{H} \to \mathscr{K}$
are approximately orthogonality preserving operators if for given $\delta, \varepsilon \in [0,1)$ one has
$$\eta\perp^{\delta} \zeta \Longrightarrow T(\eta)\perp^{\varepsilon} S(\zeta) \qquad (\eta, \zeta\in{\mathscr{H}}).$$
Often $\delta =0$ has been considered. The approximate orthogonality preserving
operators and the orthogonality equations have been investigated recently in
\cite{M.S.P, P.S.M.M, P.W, W0, Z.M.F}. Chmieli\'{n}ski \cite{J.C.4} and Turn\v{s}ek \cite{A.T.1}
studied the approximate orthogonality preserving property for one linear operator with $\delta=0$.
In addition, Chmieli\'{n}ski et al. \cite{C.L.W} verified the approximate orthogonality preserving property for two linear operators.

An inner product module over a $C^{*}$-algebra $\mathscr{A}$ is a (left)
$\mathscr{A}$-module $\mathscr{E}$ equipped with an $\mathscr{A}$-valued
inner product $\langle\cdot, \cdot\rangle$, which is linear and
$\mathscr{A}$-linear in the first variable and has the properties $\langle x, y
\rangle^*=\langle y, x\rangle$ as well as $\langle x, x\rangle \geq 0$ with equality
if and only if $x = 0$. The space $\mathscr{E}$ is called a Hilbert $\mathscr{A}$-module
if it is complete with respect to the norm $\|x\| = {\|\langle x, x\rangle\|}^{\frac{1}{2}}$.
An inner product $\mathscr{A}$-module $\mathscr{E}$ has an ``$\mathscr{A}$-valued
norm" $|\cdot|$, defined by $|x|=\langle x, x\rangle^{\frac{1}{2}}$.
By $\langle \mathscr{E}, \mathscr{E}\rangle$ we denote the closure of the span of
$\{\langle x, y\rangle: x, y \in \mathscr{E}\}$. We say that a Hilbert $\mathscr{A}$-module
$\mathscr{E}$ is full if $\langle \mathscr{E}, \mathscr{E}\rangle = \mathscr{A}$.
An isometry between inner product $\mathscr{A}$-modules $\mathscr{E}$ and $\mathscr{F}$
is an operator $U:\mathscr{E}\to \mathscr{F}$ which preserves inner products,
i.e. $\langle Ux, Uy\rangle = \langle x, y\rangle$ for all $x, y \in \mathscr{E}$.
An operator $T:\mathscr{E}\to \mathscr{F}$ is called $\mathscr{A}$-linear if it is linear
and $T(ax) = aT(x)$ for all $x\in \mathscr{E}$, $a\in\mathscr{A}$.
Further, $T$ is called local if it is linear and
$$ax = 0\, \Longrightarrow \,aT(x)= 0 \qquad (a\in\mathscr{A}, x\in \mathscr{E}).$$
Examples of local operators include multiplication and differential operators.
Note, that every $\mathscr{A}$-linear operator is local, but the converse is not true,
in general (take linear differential operators into account).
However, every bounded local operator between inner product modules is
$\mathscr{A}$-linear; see \cite{L.N.W.2}.

An operator $T:\mathscr{E}\to \mathscr{F}$ between Hilbert $\mathscr{A}$-modules $\mathscr{E}$ and $\mathscr{F}$
is called adjointable if there exists an operator $T^*:\mathscr{F}\to \mathscr{E}$
such that $\langle Tx, y\rangle = \langle x, T^*y\rangle$ for all $x\in \mathscr{E}$ and $y\in\mathscr{F}$.
It is easy to see that every adjointable operator $T$ is a bounded $\mathscr{A}$-linear operator; see \cite{Man}.

Although inner product $C^{*}$-modules generalize inner product spaces by allowing
inner products to take values in a certain $C^{*}$-algebra instead of the
$C^{*}$-algebra of complex numbers, some fundamental properties of inner product
spaces are no longer valid in inner product $C^{*}$-modules in their full generality.
For instance, they may not possess orthonormal bases or even (normalized tight) frames,
cf. \cite{Li}, and norm-closed or even orthogonally closed submodules may not be
orthogonal summands, cf. \cite{Man}.
Therefore, when we are studying inner product $C^{*}$-modules, it is always of
interest under which conditions the results analogous to those for inner product spaces
can be reobtained, as well as which more general situations might appear.
We refer the reader to \cite{Man} for more information on the basic theory of Hilbert
$C^{*}$-modules.

It is natural to explore the (approximate) orthogonality preserving property bet\-ween
inner product $C^{*}$-modules. Elements $x$ and $y$ in an inner product $C^{*}$-module
$\mathscr{E}$ are said to be orthogonal, written as
$x\perp y$, if $\langle x, y\rangle = 0$. Analogously to the Hilbert space situation,
for a given $\theta\in[0 , 1)$ two elements $x, y \in\mathscr{E}$ are approximately orthogonal or $\theta$-orthogonal,
denoted by $x\perp^{\theta} y$, if $\|\langle x, y\rangle \|
\leq \theta\|x\|\|y\|$. An operator $T: \mathscr{E} \to \mathscr{F}$ between
inner product $C^{*}$-modules is approximately orthogonality preserving if for given
$\delta, \varepsilon\in [0,1)$ one has
$$x\perp^{\delta} y \Longrightarrow T(x)\perp^{\varepsilon} T(y) \qquad (x, y\in{\mathscr{E}}).$$
This definition was introduced and investigated in \cite{I.T, M.Z}.

Two natural problems are to describe such a class of approximately orthogonality
preserving operators and to determine the stability of the orthogonality preserving property.
Let $\mathbb{K}(\mathscr{H})$ and $\mathbb{B}(\mathscr{H})$ be the
$C^{*}$-algebras of all compact linear operators and of all bounded linear operators on a
Hilbert space $\mathscr{H}$, respectively.
Recall that $\mathscr{A}$ is a standard $C^{*}$-algebra on a Hilbert space
$\mathscr{H}$ if $\mathbb{K}(\mathscr{H})\subseteq \mathscr{A} \subseteq
\mathbb{B}(\mathscr{H})$.

In the case when $\mathscr{A}$ is a standard $C^{*}$-algebra and $\delta=0$,
Ili\v{s}evi\'{c} and Turn{\v{s}}ek \cite{I.T} studied the approximate orthogonality
preserving property on $\mathscr{A}$-modules.
In \cite{M.Z}, the authors gave some sufficient conditions for a linear operator between
Hilbert $C^*$-modules to be approximately orthogonality preserving. Moreover, it was
obtained in \cite[Theorem 3.9]{M.Z}, that if $\mathscr{A}$ is a standard $C^{*}$-algebra and
$T: \mathscr{E} \to \mathscr{F}$ is a nonzero $\mathscr{A}$-linear
$(\delta, \varepsilon)$-orthogonality preserving operator between $\mathscr{A}$-modules,
then
$$\big\|\langle T(x), T(y)\rangle - \|T\|^2\langle x, y\rangle\big\|\leq \frac{4(\varepsilon -
\delta)}{(1 - \delta)(1 + \varepsilon)} \|T(x)\|\,\|T(y)\|\qquad (x, y\in \mathscr{E}).$$
Now, we will concentrate our investigations on the following condition,
$$x\perp y \Longrightarrow T(x)\perp S(y) \qquad (x, y\in\mathscr{E}),$$
which we call the orthogonality preserving property for two linear operators
$T, S: \mathscr{E}\to \mathscr{F}$.

In the case when $S = T$, the orthogonality preserving property has been treated
by Frank et al. \cite{F.M.P}, by Leung et al. \cite{L.N.W.1}, and others.

In the present paper, we show (Theorem \ref{cr.212.9}) that if $\mathscr{A}$ is a standard $C^{*}$-algebra
and $T, S:\mathscr{E}\to \mathscr{F}$ are two nonzero local operators
between inner product $\mathscr{A}$-modules, then $x \perp y\, \Longrightarrow \,
T(x) \perp S(y)$ for all $x, y\in \mathscr{E}$ if and only if there exists $\gamma\in
\mathbb{C}$ such that $\langle T(x), S(y)\rangle = \gamma \langle x, y\rangle$ for all
$x, y\in \mathscr{E}$. In particular, $T$ and $S$ are $\mathscr{A}$-linear. In fact, this result can be
considered as a generalization of Theorem \ref{th.1.1}.
We then apply it in Theorem \ref{Thm_3.1} to prove that if ${\mathscr A}$ is a $C^*$-algebra and
$T, S : {\mathscr E} \to {\mathscr F}$ are two nonzero
bounded ${\mathscr A}$-linear operators between
full Hilbert ${\mathscr A}$-modules such that
$x \perp y\, \Longrightarrow \,T(x) \perp S(y)$ for all $x, y\in \mathscr{E}$,
then there exists an element $\gamma$ of the center $Z(M({\mathscr A}))$ of the multiplier
algebra $M({\mathscr A})$ of ${\mathscr A}$ such that $\langle T(x), S(y)\rangle =
\gamma \langle x, y\rangle$ for all $x, y\in \mathscr{E}$.
In the case of pairs of merely bounded linear operators $S$ and $T$,
the invertibility of $S$ implies the ${\mathscr A}$-linearity and adjointability of these operators, and $T=(S^*)^{-1}$.

\section{Linear and local orthogonality-preserving operators}
The aim of this section is to prove an analogue of Theorem \ref{th.1.1} for
two unknown linear operators in inner product $C^*$-modules, and subsequently,
a generalization of Theorem \ref{th.1.1} for $\mathscr{A}$-linear
operators between Hilbert $\mathscr{A}$-modules.
Let us start with some observations. The following result is a consequence of
\cite[Theorem 3.1]{A.R} and \cite[Lemma 4.1]{Z.M.F}.

\begin{proposition}\label{pr.212.5}
Let $\mathscr{E}$ and $\mathscr{F}$ be two inner product $\mathscr{A}$-modules
and $x, y\in \mathscr{E}$. Let $T, S:\mathscr{E}\to \mathscr{F}$ be
two nonzero operators. The following statements are mutually equivalent:
\begin{itemize}
\item[(i)] $x \perp y\, \Longrightarrow \,T(x) \perp S(y)$.
\item[(ii)] $|x - \lambda y| = |x + \lambda y|\, \Longrightarrow \,|T(x) - \lambda S(y)|
= |T(x) + \lambda S(y)|$ for all $\lambda\in \mathbb{C}$.
\item[(iii)] $|x - ay| = |x + ay|\, \Longrightarrow \,|T(x) - aS(y)| = |T(x) + aS(y)|$ for
all $a\in \mathscr{A}$.
\item[(iv)] $|x|^2 \leq |x + \lambda y|^2\, \Longrightarrow \,|T(x)|^2 \leq |T(x) +
\lambda S(y)|^2$ for all $\lambda\in \mathbb{C}$.
\item[(v)] $|x|^2 \leq |x + ay|^2\, \Longrightarrow \,|T(x)|^2 \leq |T(x) + aS(y)|^2$
for all $a\in \mathscr{A}$.
\item[(vi)] $|x| \leq |x + ay|\, \Longrightarrow \,|T(x)| \leq |T(x) + aS(y)|$ for all $a\in \mathscr{A}$.
\end{itemize}
\end{proposition}

\begin{remark}
For inner product $\mathscr{A}$-modules $\mathscr{E}$ and $\mathscr{F}$
and nonzero operators $T, S:\mathscr{E}\to \mathscr{F}$, we do not know whether
the following statements for $x, y\in \mathscr{E}$ are mutually equivalent:
\begin{itemize}
\item[(i)] $x \perp y\, \Longrightarrow \,T(x) \perp S(y)$.
\item[(ii)] $|x| \leq |x + \lambda y|\, \Longrightarrow \,|T(x)| \leq |T(x) +
\lambda S(y)|$ for all $\lambda\in \mathbb{C}$.
\end{itemize}
Now, we consider the $C^*$-algebra $\mathbb{M}_2(\mathbb{C})$ of all
complex $2\times 2$ matrices, as an inner product
$C^*$-module over itself. Let $A, B\in \mathbb{M}_2(\mathbb{C})$ and let
$T, S:\mathbb{M}_2(\mathbb{C})\to \mathbb{M}_2(\mathbb{C})$
be two nonzero operators. Then, by \cite[Proposition 3.6]{A.R}, the following
statements are mutually equivalent:
\begin{itemize}
\item[(i)] $A \perp B\, \Longrightarrow \,T(A) \perp S(B)$.
\item[(ii)] $|A| \leq |A + \lambda B|\, \Longrightarrow \,|T(A)| \leq |T(A) + \lambda S(B)|$
for all $\lambda\in \mathbb{C}$.
\end{itemize}
\end{remark}

Employing the polarization identity, we obtain the next result.
\begin{proposition}\label{pr.212.6}
Let $\mathscr{E}$ and $\mathscr{F}$ be two inner product $\mathscr{A}$-modules.
Let $T, S:\mathscr{E}\to \mathscr{F}$ be two nonzero linear operators
such that $\langle T(x), S(x)\rangle = \gamma |x|^2$ for all $x\in \mathscr{E}$ and
for some $\gamma\in Z(M(\mathscr{A}))$. Then
$$x \perp y\, \Longrightarrow \,T(x) \perp S(y) \qquad (x, y\in \mathscr{E}).$$
\end{proposition}

Notice that the converse of the above proposition is not true, even in the case $T = S$;
see \cite[Example 3.14]{M.Z}. In the next theorem, we prove that the converse of
the above proposition is true if $\mathscr{A}$ is a standard $C^{*}$-algebra,
in particular, whenever $Z(M({\mathscr{A}}))=\mathbb{C}$.

To achieve the next theorem we state some prerequisites.
Given two vectors $\eta$  and $\zeta$ in a Hilbert space ${\mathscr{H}}$, we shall
denote the one-rank operator defined by $(\eta\otimes \zeta)(\xi) = [\xi, \zeta]\eta$
 by $\eta\otimes \zeta \in \mathbb{K}(\mathscr{H})$.
Observe that $\eta\otimes \eta$ is a minimal projection. Recall that a projection
$e$ in a $C^*$-algebra $\mathscr{A}$ is called minimal if $e \mathscr{A}e =\mathbb{C}e$.
Now let $\mathscr{A}$ be a standard $C^{*}$-algebra on a Hilbert space $\mathscr{H}$
and let $\mathscr{E}$ be an inner product (respectively, Hilbert) $\mathscr{A}$-module.
Let $e = \eta\otimes \eta$ for some unit vector $\eta\in {\mathscr{H}}$ be a minimal
projection. Then $\mathscr{E}_e = \{ex: \, x\in\mathscr{E}\}$ is a complex inner product
(respectively, Hilbert) space contained in $\mathscr{E}$ with respect to the inner product
$[x, y] = {\rm tr}(\langle x, y\rangle)$, $x, y\in {\mathscr{E}}_e$; see \cite{B.G}. Note that if $x, y\in
\mathscr{E}_e$, then $\langle x, y\rangle = [x, y]e$ and ${\|x\|}_{\mathscr{E}_e} =
{\|x\|}_{\mathscr{E}}$, where the norm ${\|.\|}_{\mathscr{E}_e}$ comes from the inner
product $[\cdot,\cdot]$. This enables us to apply Hilbert space theory by lifting results from the
Hilbert space $\mathscr{E}_e$ to the whole $\mathscr{A}$-module $\mathscr{E}$.

\begin{theorem}\label{th.212.7}
Let $\mathscr{A}$ be a standard $C^{*}$-algebra on a Hilbert space $\mathscr{H}$
and let $\mathscr{E}$ and $\mathscr{F}$ be two inner product $\mathscr{A}$-modules.
Suppose, $T, S:\mathscr{E}\to \mathscr{F}$ are two nonzero $\mathscr{A}$-linear operators such that
$$x \perp y\, \Longrightarrow \,T(x) \perp S(y) \qquad (x, y\in \mathscr{E}).$$
Then there exists $\gamma\in \mathbb{C}$ such that
$$\langle T(x), S(y)\rangle = \gamma \langle x, y\rangle \qquad (x, y\in \mathscr{E}).$$
\end{theorem}

\begin{proof}
The following proof is a modification of the one given by Ili\v{s}evi\'{c} and
Turn{\v{s}}ek \cite[Theorem 3.1]{I.T}. Let $e = \zeta\otimes\zeta$ and $f = \eta\otimes\eta$ be minimal projections
in $\mathscr{A}$ and let $u = \zeta\otimes\eta$.
Also, let $T_e= T_{|_{\mathscr{E}_e}}$ and $S_e= S_{|_{\mathscr{E}_e}}$.
For linear operators $T_e, S_e: \mathscr{E}_e
\to \mathscr{F}_e$ we have $[x, y] = 0\, \Longrightarrow
\,[{T_e}(x), {S_e}(y)] = 0$ for all $x, y\in \mathscr{E}_e$.
Hence, by Theorem \ref{th.1.1}, there exists $\gamma_e\in \mathbb{C}$ such that
$$[T(ex), S(ex)] = \gamma_e \|ex\|^2 = \gamma_e [ex, ex] \qquad (x\in \mathscr{E}_e).$$
This yields $[T(ex), S(ex)]e = \gamma_e [ex, ex]e$, thus
$$\langle T(ex), S(ex)\rangle
= \gamma_e \langle ex, ex\rangle = \gamma_e |ex|^2,$$
or equivalently,
\begin{align}\label{id.212.71}
e\langle T(x), S(x)\rangle e = \gamma_e e|x|^2e \qquad (x\in \mathscr{E}).
\end{align}
Similarly, there exists $\gamma_f\in \mathbb{C}$ such that
\begin{align}\label{id.212.72}
f\langle T(x), S(x)\rangle f = \gamma_f f|x|^2f \qquad (x\in \mathscr{E}).
\end{align}
Since $ufu^* = e$, it follows from (\ref{id.212.71}) and (\ref{id.212.72}) that
\begin{align*}
\gamma_e [ex, ex]e& = \gamma_e \langle ex, ex\rangle = \gamma_e e\langle x, x\rangle e
\\& = e\langle T(x), S(x)\rangle e = ufu^* \langle T(x), S(x)\rangle ufu^*
\\& = uf \langle T(u^*x), S(u^*x)\rangle fu^* = u \gamma_f f|u^*x|^2f u^*
\\&= \gamma_f ufu^* |x|^2 ufu^* = \gamma_f e\langle x, x\rangle e
\\& = \gamma_f \langle ex, ex\rangle = \gamma_f [ex, ex]e.
\end{align*}
Thus
$$\gamma_e [ex, ex] = \gamma_f [ex, ex] \qquad (x\in \mathscr{E}).$$
Replacing $x$ with $\frac{x}{\|ex\|}$, we conclude $\gamma_e = \gamma_f = \gamma$.
Hence, by (\ref{id.212.71}), we get
$$e\langle T(x), S(x)\rangle e = e\gamma|x|^2e \qquad (x\in \mathscr{E})$$
for all minimal projections $e \in \mathscr{A}$. Having in mind that $\mathscr{A}$ is a
standard $C^*$-algebra, we deduce that
$$\langle T(x), S(x)\rangle = \gamma |x|^2 \qquad (x\in \mathscr{E}).$$
Now, by the polarization identity, we arrive at
$$\langle T(x), S(y)\rangle = \gamma \langle x, y\rangle \qquad (x, y\in \mathscr{E}).$$
\end{proof}
\begin{corollary}\label{cr.212.10}
Let $\mathscr{A}$ be a standard $C^{*}$-algebra and let
$\{\mathscr{E}, {\langle \cdot, \cdot\rangle}_1\}$ be an inner product $\mathscr{A}$-module.
Suppose that ${\langle \cdot, \cdot\rangle}_2$
is a second $\mathscr{A}$-valued inner product on $\mathscr{E}$.
If ${\langle x, y\rangle}_1 = 0$ implies ${\langle x, y\rangle}_2 = 0$ for every
$x, y\in \mathscr{E}$, then there exists a positive constant
$\gamma\in \mathbb{C}$ such that ${\langle x, y\rangle}_2 = \gamma {\langle x, y\rangle}_1$
for each $x, y\in \mathscr{E}$.
\end{corollary}

\begin{proof}
Take $\mathscr{E} = \mathscr{F}$ as $\mathscr{A}$-modules and set $T = S =I:
(\mathscr{E}, {\langle \cdot,\cdot\rangle}_1) \to (\mathscr{E}, {\langle ., .\rangle}_2)$, where $I$
denotes the identity operator. Applying Theorem \ref{th.212.7} the assertion follows.
\end{proof}
\begin{remark}
According to \cite[Example 3.15]{M.Z}, the assumption of $\mathscr{A}$-linearity,
even in the case $T = S$, is necessary in Theorem \ref{th.212.7}. If either $S$ or $T$ is
adjointable and $\gamma \not= 0$, then ${\rm ran}(S)={\rm ran(T)}=\mathscr{E}$,
because e.g. $S^*T=\gamma I$ for a real number $\gamma > 0$.
\end{remark}

In the following result, we employ some ideas of \cite{L.N.W.1} to consider local operators between
inner product $\mathscr{A}$-modules, i.e. linear operators $T: \mathscr{E} \to \mathscr{F}$ such
that $ax=0$ for $x\in\mathscr{E}$ and $a\in\mathscr{A}$ forces $aT(x)=0$ in $\mathscr{F}$. We are
interested in operators which preserve orthogonality. Since we do not suppose that the local operators
under consideration are bounded, the $\mathscr{A}$-linearity should be obtained separately.
\begin{theorem}\label{cr.212.8}
Let $\mathscr{A}$ be a standard $C^{*}$-algebra on a Hilbert space $\mathscr{H}$
and let $\mathscr{E}$ and $\mathscr{F}$ be two inner product $\mathscr{A}$-modules.
Suppose that $T, S:\mathscr{E}\to \mathscr{F}$ are two nonzero
local operators such that
$$x \perp y\, \Longrightarrow \,T(x) \perp S(y) \qquad (x, y\in \mathscr{E}).$$
Then there exists $\gamma\in \mathbb{C}\setminus \{ 0 \}$ such that
$$\langle T(x), S(y)\rangle = \gamma \langle x, y\rangle \qquad (x, y\in \mathscr{E}).$$
Moreover, the operators $T$ and $S$ are $\mathscr{A}$-linear.
\end{theorem}

\begin{proof}
Let $(f_j)_{j\in J}$ be an approximate unit in $\mathbb{K}(\mathscr{H})$ consisting of finite
rank positive operators. Suppose that $p\in \mathbb{K}(\mathscr{H})$
is a projection. Since $T$ is local, the known equality $p(1-p)x=0$ ensures $pT((1-p)x)=0$, and
analogously, we get $(1-p)T(px)=0$. From the complex linearity of the operator $T$ we derive $pT(x)=T(px)$
from these two equalities. Consequently, $T(ax) = aT(x)$ for all finite rank operators $a \in
\mathscr{A}$ and all $x\in \mathscr{E}$.
Now, for each $x\in \mathbb{K}(\mathscr{H})\cdot \mathscr{E}$, there exist $c\in
\mathbb{K}(\mathscr{H})$ and $z\in \mathscr{E}$ such that $x = cz$.
Hence
$$\lim\limits_{j} \big\|f_jT(x) - cT(z) \big\| = \lim\limits_{j} \big\|f_jcT(z) - cT(z) \big\| = 0.$$
Define $\widetilde{T}: \mathbb{K}(\mathscr{H})\cdot \mathscr{E}
\to \mathbb{K}(\mathscr{H})\cdot \mathscr{F}$ by setting
$\widetilde{T}(x)$ to be the norm limit of $f_jT(x)$.
Notice that $\lim\limits_{j} \big\|f_jT(x) - cT(z) \big\| = 0$
shows that $\widetilde{T}(x)$ depend on neither the choice of $(f_j)_{j\in J}$ nor the decomposition $x = cz$. Also, $\widetilde{T}$ is $\mathbb{K}(\mathscr{H})$-linear since,
$$\widetilde{T}(ax) = \widetilde{T}(acz) = acT(z) = a\widetilde{T}(x)$$
for all $x\in \mathscr{E}$ and all $a\in \mathbb{K}(\mathscr{H})$.
Similarly, define $\widetilde{S}: \mathbb{K}(\mathscr{H})\cdot \mathscr{E}
\to \mathbb{K}(\mathscr{H})\cdot \mathscr{F}$ by setting
$\widetilde{S}(y)$ to be the norm limit of $f_jS(y)$.
Now, if $x, y\in \mathbb{K}(\mathscr{H})\cdot \mathscr{E}$ with $\langle x, y\rangle = 0$, then
$\langle T(x), S(y)\rangle = 0$, which secures $\langle f_jT(x), f_kS(y)\rangle = 0$ for all $j,
k\in J$. Thus $\langle \widetilde{T}(x), \widetilde{S}(y)\rangle = 0$. Hence for
$\mathbb{K}(\mathscr{H})$-linear operators $\widetilde{T}$ and $\widetilde{S}$ we have
$$x \perp y\, \Longrightarrow \,\widetilde{T}(x)\perp \widetilde{S}(y)
\qquad (x, y\in \mathbb{K}(\mathscr{H})\cdot \mathscr{E}).$$
So, by Theorem \ref{th.212.7}, there exists $\gamma\in \mathbb{C}$ such that
$\langle \widetilde{T}(x), \widetilde{S}(x)\rangle = \gamma |x|^2$ for all
$x\in \mathbb{K}(\mathscr{H})\cdot \mathscr{E}$. Thus
$$f_j\langle T(x), S(x)\rangle f_j = \langle \widetilde{T}(f_jx), \widetilde{S}(f_jx)\rangle
= \gamma |f_jx|^2 = f_j\gamma |x|^2 f_j$$
for all $x\in \mathscr{E}$ and all $j\in J$.
Consequently, if $\langle T(x), S(x)\rangle - \gamma |x|^2 \in \mathscr{A}$,
then always $f_j\big(\langle T(x), S(x)\rangle - \gamma |x|^2\big)f_j=0$,
which yields $\langle T(x), S(x)\rangle - \gamma |x|^2 =0$.
Hence, $\langle T(x), S(x)\rangle = \gamma |x|^2$ for all $x\in \mathscr{E}$ . Utilizing the polarization identity, we obtain
$$\langle T(x), S(y)\rangle = \gamma \langle x, y\rangle \qquad (x, y\in \mathscr{E}).$$
\end{proof}

Combining Proposition \ref{pr.212.6} and Theorem \ref{cr.212.8}, we reach the next
result. In fact, this result is a generalization of \cite[Theorem 3.1]{I.T} and
\cite[Theorem 4.10]{Z.M.F}. The result also generalizes \cite[Corollary 3.2]{L.N.W.1}.

\begin{theorem}\label{cr.212.9}
Let $\mathscr{A}$ be a standard $C^{*}$-algebra and let $\mathscr{E}$ and $\mathscr{F}$
be two inner product $\mathscr{A}$-modules. Suppose that $T, S:\mathscr{E} \to
\mathscr{F}$ are two nonzero local operators. The following statements are
mutually equivalent:
\begin{itemize}
\item[(i)] $x \perp y\, \Longrightarrow \,T(x) \perp S(y)$ for all $x, y\in \mathscr{E}$.
\item[(ii)] There exists $\gamma\in \mathbb{C}\setminus \{ 0 \}$ such that $\langle T(x), S(x) \rangle =
\gamma |x|^2$ for all $x\in \mathscr{E}$.
\item[(iii)] There exists $\gamma\in \mathbb{C}\setminus \{ 0 \}$ such that $\langle T(x), S(y) \rangle =
\gamma \langle x, y\rangle$ for all $x, y\in \mathscr{E}$.
\end{itemize}
Under these conditions the operators $T$ and $S$ are $\mathscr{A}$-linear.
\end{theorem}

\begin{corollary} \label{Cor_2.10}
Let $\mathscr{A}$ be a standard $C^{*}$-algebra and let $\mathscr{E}$ be an inner
product $\mathscr{A}$-module. Suppose that $T:\mathscr{E} \to \mathscr{E}$ is a
nonzero local operator such that
$$x \perp y\, \Longrightarrow \,T(x) \perp y \qquad (x, y\in \mathscr{E}).$$
Then there exists $\gamma\in \mathbb{C}\setminus \{ 0 \}$ such that $T(x) = \gamma x$ for all $x\in
\mathscr{E}$.
\end{corollary}

\begin{proof}
Applying Theorem \ref{cr.212.9} to $T$ and $S = I$ we obtain, with some $\gamma
\in \mathbb{C}$, $\langle T(x), y\rangle = \gamma\langle x, y\rangle$ for all
$x, y\in \mathscr{E}$. Hence, $\langle T(x) - \gamma x, y\rangle = 0$ for all $x, y\in
\mathscr{E}$. Thus $T(x) = \gamma x$ for all $x\in \mathscr{E}$.
\end{proof}
\begin{corollary}
Let $\mathscr{A}$ be a standard $C^{*}$-algebra and let $\mathscr{E}$ and
$\mathscr{F}$ be two inner product $\mathscr{A}$-modules. Let
$T_0, S_0:\mathscr{E} \to \mathscr{F}$ be two nonzero local operators such that
$$x \perp y\, \Longrightarrow \,T_0(x) \perp S_0(y) \qquad (x, y\in \mathscr{E}).$$
Suppose, the linear operators $T, S:\mathscr{E} \to \mathscr{F}$ are sufficiently close to
$T_0$ and $S_0$, respectively, namely that for $\theta_1, \theta_2\in[0, 1)$
and for all $x, y\in \mathscr{E}$
$$\|T(x) - T_0(x)\|\leq \theta_1 \|T(x)\| \qquad \mbox{and}\qquad
\|S(y) - S_0(y)\|\leq \theta_2 \|S(y)\|.$$
Then
$$x\perp y\, \Longrightarrow \,T(x)\perp^\varepsilon S(y) \qquad (x, y\in \mathscr{E}),$$
where $\varepsilon = \theta_1\theta_2 + \theta_1(\theta_2 + 1) + (\theta_1 + 1)\theta_2$.
\end{corollary}

\begin{proof}
From the assumption, we obtain
\begin{align}\label{id.212.711}
\|T_0(x)\|\leq (\theta_1 + 1)\|T(x)\| \quad \mbox{and}\quad
\|S_0(y)\|\leq (\theta_2 + 1)\|S(y)\| \qquad (x, y\in \mathscr{E}).
\end{align}
Also, by Theorem \ref{cr.212.9}, there exists $\gamma_0\in \mathbb{C}$ such that
\begin{align}\label{id.212.712}
\langle T_0(x), S_0(y)\rangle = \gamma_0 \langle x, y\rangle \qquad (x, y\in \mathscr{E}).
\end{align}
Now let $x, y\in \mathscr{E}$ and $x\perp y$. From (\ref{id.212.711}) and (\ref{id.212.712})
we get
\begin{align*}
\|\langle T(x), S(y)\rangle\| & =\|\langle T(x), S(y)\rangle - \langle T_0(x), S_0(y)\rangle\|
\\& = \big\|\langle T(x) - T_0(x), S(y) - S_0(y)\rangle + \langle T(x) - T_0(x), S_0(y)\rangle
\\& \qquad + \langle T_0(x), S(y) - S_0(y)\rangle\big\|
\\& \leq \|T(x) - T_0(x)\|\,\|S(y) - S_0(y)\| + \|T(x) - T_0(x)\|\,\|S_0(y)\|
\\& \qquad + \|T_0(x)\|\, \|S(y) - S_0(y)\|
\\& \leq \Big(\theta_1\theta_2 + \theta_1(\theta_2 + 1) + (\theta_1 + 1)\theta_2\Big)
\|T(x)\|\,\|S(y)\|
\\& = \varepsilon \|T(x)\|\,\|S(y)\|.
\end{align*}
Thus $\|\langle T(x), S(y)\rangle\| \leq \varepsilon \|T(x)\|\,\|S(y)\|$, and hence $T(x)
\perp^\varepsilon S(y)$.
\end{proof}
\section{$C^*$-linear orthogonality preserving operators}
In this section, we intend to show the properties of pairs of bounded $C^*$-linear
operators $\{T,S\}$ for which the orthogonality of two elements $x$ and $y$ of
the domain ensures the orthogonality of their
respective images $T(x)$ and $S(y)$. To get reasonable results we have
either to suppose or to derive the $C^*$-linearity of the operators. The proof
of the key equality relies, e.g., on the theory of the universal $*$-representation of
the $C^*$-algebra of coefficients, in which the bicommutant of the faithfully
represented $C^*$-algebra is $*$-isomorphic to its bidual Banach space.
Also, we make use of the existence of predual spaces of von Neumann
algebras and for self-dual Hilbert $C^*$-modules over them. For a
concise explanation we refer the reader to \cite[Ch. II, 2+3; Ch. V, 1]{Takesaki} and
to \cite[{\S\S} 3+4]{Pa1973}.

We need some prerequisites for the next theorem; see \cite{Frank:ZAA, F.M.P}.
For a Hilbert $\mathscr{A}$-module $\mathscr{E}$ over a $C^*$-algebra
$\mathscr{A}$, one can extend ${\mathscr E}$ canonically to a Hilbert
$\mathscr{A}^{**}$-module ${\mathscr E}^{\#}$ over the bidual Banach space
and von Neumann algebra $\mathscr{A}^{**}$ of $\mathscr{A}$
\cite[Theorem 3.2, Proposition 3.8, and {\S }4]{Pa1973}. For this the
$\mathscr{A}^{**}$-valued pre-inner product can be defined by the formula
     $$[a\otimes x,b\otimes y]=a^*\langle x,y\rangle b,$$
for elementary tensors of $\mathscr{A}^{**}\otimes {\mathscr E}$, where
$a,b\in \mathscr{A}^{**}$, $x,y \in {\mathscr E}$. The quotient module
of $\mathscr{A}^{**}\otimes \mathscr{E}$ by the
set of all isotropic vectors is denoted by ${\mathscr E}^{\#}$.
Denote the canonical isometric module embedding of $\mathcal E$ into
${\mathcal E}^{\#}$  described in \cite{Pa1973} by $\pi'$.
The quotient module can be canonically completed to a self-dual Hilbert
$\mathscr{A}^{**}$-module $\mathscr G$ which is isometrically algebraically
isomorphic to the $\mathscr{A}^{**}$-dual $\mathscr{A}^{**}$-module of
${\mathscr E}^{\#}$.  In addition, $\mathscr G$ is a dual Banach space itself;
cf. \cite[Theorem 3.2, Proposition 3.8, and {\S}4]{Pa1973}.
Every $\mathscr{A}$-linear bounded operator $T: {\mathscr E} \to {\mathscr E}$
can be continued to a unique $\mathscr{A}^{**}$-linear operator $T: {\mathscr E}^{\#}
\to {\mathscr E}^{\#}$ preserving the operator norm and obeying the
canonical embedding $\pi'({\mathscr E})$ of
$\mathscr E$ into ${\mathscr E}^{\#}$. Similarly, $T$ can be
further extended to the self-dual Hilbert ${\mathscr A}^{**}$-module $\mathscr G$.
The extension is such that the isometrically algebraically embedded
copy $\pi'({\mathscr E})$ of $\mathscr E$ in $\mathscr G$ is a
$w^*$-dense $\mathscr{A}$-submodule of $\mathscr G$, and that $\mathscr A$-valued inner
product values of elements of $\mathscr E$ embedded in $\mathscr G$
are preserved with respect to the $\mathscr{A}^{**}$-valued inner product on
$\mathscr G$ and to the canonical isometric embedding $\pi$ of $\mathscr A$
into its bidual Banach space ${\mathscr A}^{**}$. Every bounded $\mathscr A$-linear
operator $T$ on $\mathscr E$ can be extended to a unique bounded
$\mathscr{A}^{**}$-linear operator on $\mathscr G$ preserving the operator
norm, cf. \cite[Proposition 3.6, Corollary 3.7, and {\S}4]{Pa1973}. The extension
of bounded $\mathscr A$-linear operators from $\mathscr E$ to $\mathscr G$
is continuous with respect to the $w^*$-topology on ${\mathscr G}$.
A Hilbert $C^*$-module $\mathscr K$ over a $W^*$-algebra $\mathscr B$ is self-dual,
if and only if its unit ball is complete with respect to the topology
induced by the semi-norms $\{ |f(\langle x,\cdot\rangle)| : x \in
\mathscr K, f \in {\mathscr B}^*, \|x\| \leq 1, \|f\| \leq 1 \}$,
if and only if its unit ball is complete with respect to the topology
induced by the semi-norms $\{ {f(\langle \cdot,\cdot \rangle)}^{1/2} :
f \in {\mathscr B}^*, \|f\| \leq 1\}$. The first topology coincides
with the $w^*$-topology on $\mathscr K$ in that case, see \cite[{\S} 4]{Frank:ZAA}.

Note, that in the construction above $\mathscr E$ is always $w^*$-dense
in $\mathscr G$, as well as for each subset of $\mathscr E$ the
respective construction is $w^*$-dense in its biorthogonal
complement with respect to $\mathscr G$. However, starting
with a subset of $\mathscr G$, its biorthogonal complement
with respect to $\mathscr G$ might not have a $w^*$-dense
intersection with the embedding of $\mathscr E$ into $\mathscr G$;
cf. \cite[Proposition 3.11.9]{Ped1979}.

Further, we want to consider only discrete $W^*$-algebras,
i.e. $W^*$-algebras for which the supremum of all minimal projections
contained in them equals their identity. (We prefer to use the word
discrete instead of atomic.) To connect to the general $C^*$-case
we make use of a theorem of Akemann \cite[p. 278]{Ak1969} stating that the
$*$-homomorphism of a $C^*$-algebra ${\mathscr A}$ into the discrete part of
its bidual von Neumann algebra ${\mathscr A}^{**}$ which arises as the
composition of the canonical embedding $\pi$ of ${\mathscr A}$ into ${\mathscr A}^{**}$
followed by the projection to the discrete part of ${\mathscr A}^{**}$
is an injective $*$-homomorphism $\rho$. This injective $*$-homomorphism $\rho$ is
implemented by a central projection $p \in Z({\mathscr A}^{**})$ in such a
way that ${\mathscr A}^{**}$ multiplied by $p$ gives the discrete part of ${\mathscr A}^{**}$.

For topological characterizations of self-duality of Hilbert $C^*$-modules over $W^*$-algebras
and the properties of the modules and operators we refer the reader to
\cite{Pa1973}.

Now, we are in a position to state one of the main results of this section that
generalizes \cite[Theorem 3.2]{Ch.new}.

\begin{theorem} \label{Thm_3.1}
Let ${\mathscr A}$ be a $C^*$-algebra and let ${\mathscr E}$ and ${\mathscr F}$
be two full Hilbert ${\mathscr A}$-modules.
Suppose that $T, S : {\mathscr E} \to {\mathscr F}$ are two nonzero
bounded ${\mathscr A}$-linear operators such that
$$x \perp y\, \Longrightarrow \,T(x) \perp S(y) \qquad (x, y\in \mathscr{E}).$$
Then there exists an element $\gamma$ of the center $Z(M({\mathscr A}))$ of
the multiplier algebra $M({\mathscr A})$ of ${\mathscr A}$ such that
$$\langle T(x), S(y)\rangle = \gamma \langle x, y\rangle \qquad (x, y\in \mathscr{E}).$$
In the complementary case, if $\gamma=0$, the ranges of the operators $T$ and $S$
are orthogonal to each other, and no further information on them can be derived.
\end{theorem}

\begin{proof}
First, we make use of the existing canonical non-degenerate isometric $*$-representation
$\pi$ of a $C^*$-algebra ${\mathscr A}$ in its bidual Banach space and von Neumann
algebra ${\mathscr A}^{**}$ of ${\mathscr A}$, as well as of its extension
$\pi': {\mathscr E} \to {\mathscr E}^{\#} \to {\mathscr G}$,
$\pi': {\mathscr F} \to {\mathscr F}^{\#} \to {\mathscr H}$
and of the unique $w^*$-continuous ${\mathscr A}^{**}$-linear bounded
operator extensions $T,S: {\mathscr G} \to {\mathscr H}$. In other words,
we extend the set $\{ {\mathscr A}, {\mathscr E}, {\mathscr F}, T, S \}$
to the set $\{ {\mathscr A}^{**}, {\mathscr G}, {\mathscr H}, T, S \}$. The
Hilbert ${\mathscr A}^{**}$-modules ${\mathscr G}$ and ${\mathscr H}$ are
self-dual and admit a predual Banach space, hence, a $w^*$-topology. The
extended operators $T$ and $S$ are $w^*$-continuous, ${\mathscr A}^{**}$-linear
and bounded by the same constants as the original operators $T$ and $S$.

Secondly, we have to demonstrate that for each pair of elements
$x \perp y$ with $x,y \in {\mathscr G}$ the property $T(x) \perp S(y)$
still holds for the extended operators. In the sequel, we identify $\mathscr E$ with
its image $\pi'({\mathscr E}) \subseteq \mathscr G$. Note, that $\pi'(\mathscr E)$ is $w^*$-dense in $\mathscr G$.
Applying techniques developed in \cite{Pa1973},
let us fix a non-trivial normal state $f \in \mathscr{A}^*$ and form the inner product spaces
${\mathscr E}_f=\{ \mathscr E, f(\langle \cdot,\cdot \rangle) \} / {\ker}(f(\langle \cdot,\cdot \rangle))$ and
${\mathscr G}_f={\rm cl}(\{ \mathscr G, f(\langle \cdot,\cdot \rangle) \} / {\ker}(f(\langle \cdot,\cdot \rangle)))$,
where ${\mathscr E}_f$ is norm-dense in ${\mathscr G}_f$. The operators $T$ and $S$ can
be restricted to this pre-Hilbert space and are still linear and continuous. Now, form the biorthogonal complements
of the two remainder classes $[x]_f$ and $[y]_f$ of $x, y \in{\mathscr G}_f$. Their intersection with
${\mathscr E}_f$ is not empty, and these intersections form norm-dense subsets in them, respectively.
For a pair of elements $t \in [x]_f \cap {\mathscr E}_f$ and $s \in [y]_f \cap {\mathscr E}_f$ we always
have $T(t) \perp_f S(s)$. Since inner products on pre-Hilbert spaces are separately weakly continuous
in each of their two arguments the domain of $T$, for instance, can be extended to ${\mathscr G}_f$
preserving the orthogonality relation to $S({\mathscr E}_f)$. By symmetry, the same is true for $S$.
Therefore, $T([x]_f) \perp_f S([y]_f)$ for all $x,y \in {\mathscr G}_f$ because of the norm-density of
the respective subsets and of the continuity of the operators. Since the normal state $f$ on $\mathscr{A}$
was selected arbitrarily the relation $T(x) \perp S(y)$ follows.

Since the von Neumann algebra $p{\mathscr A}^{**}$ is discrete, its identity
$p$ can be represented as the $w^*$-sum of a maximal set of pairwise
orthogonal atomic projections $\{ q_\alpha : \alpha \in J\}$ of the center
$Z(p{\mathscr A}^{**})$ of $p{\mathscr A}^{**}$. Note, that $w^*$-$\sum_{\alpha
\in J} q_\alpha = p$. Take a single atomic projection $q_\alpha \in
Z(p{\mathscr A}^{**})$ of this collection and consider the part
$\big\{q_\alpha p{\mathscr A}^{**}, q_\alpha p{\mathscr E}, q_\alpha p{\mathscr F}, q_\alpha pT, q_\alpha pS\big\}$
of the problem.

Since $q_\alpha p{\mathscr A}^{**}$ is a discrete (type I) $W^*$-factor (finite-
or infinite-dimensional), the equality $\langle T(x),S(y) \rangle = \lambda_{q_\alpha}
\langle x,y \rangle = \langle x,y \rangle \lambda_{q_\alpha}$ holds for
specific $T$ and $S$ and complex number $\lambda_{q_\alpha}$, cf. Theorem \ref{th.212.7}.

Now, we can follow the decomposition process in reverse direction.
Note, that the multiplier algebra of $p{\mathscr A}$ is $*$-isometrically
embedded in $p{\mathscr A}^{**}$.
Since $|\lambda_{q_\alpha}| \leq \|S\|\,\|T\|$ for all $\alpha \in J$, the
sum $\sum_{\alpha \in J} \lambda_{q_\alpha} q_\alpha$ is $w^*$-convergent
in $Z(p{\mathscr A}^{**})$. Moreover, since $\lambda_{q_\alpha} q_\alpha$ commutes with
$p\langle x,y \rangle$ for each $\alpha$, the sum $\sum_{\alpha \in J}
\lambda_{q_\alpha} q_\alpha$ commutes with $p\langle x,y \rangle$. What is more,
since $p\langle {\mathscr G}, {\mathscr G}\rangle$ is dense in $p{\mathscr A}^{**}$
and the $C^*$-valued inner product $\sum_{\alpha \in J} \lambda_{q_\alpha}
q_\alpha \langle x,y \rangle =
\sum_{\alpha \in J} \langle x,y \rangle \lambda_{q_\alpha} q_\alpha$
belongs to $p{\mathscr A}^{**}$ for all $x,y \in {\mathscr G}$, the element
$\sum_{\alpha \in J} \lambda_{q_\alpha} q_\alpha$ is in $p Z(M({\mathscr A}))$.
Since ${\mathscr E}$ and ${\mathscr F}$ are full Hilbert ${\mathscr A}$-modules,
we arrive at the equality $p\langle T(x),S(y) \rangle= \sum_{\alpha \in J}
\lambda_{q_\alpha} q_\alpha \langle x,y \rangle $ for all $x,y \in {\mathscr G}$.

The remaining step is to pull back this equality to the initial context along
the two injective $*$-homomorphisms used.

Note, that $\gamma = 0$ forces only $T({\mathscr E}) \perp S({\mathscr E})$
without further conditions on $T$ and $S$.
\end{proof}

\begin{example} \label{Ex_extra}
Let $\mathscr{A}=C_0((0,1])$ be the $C^*$-algebra of all continuous functions
vanishing at zero, where $(0,1]$ is the unit interval with the usual metric
topology. Set $\mathscr{E}=\mathscr{F}=\mathscr{A}$ with the usual $\mathscr{A}$-valued
inner product derived from the multiplication and the involution on $\mathscr{A}$. Set
$$S(f(t)) = \sin \left(\frac{1}{t}\right) \cdot f(t) \quad , \quad T(f(t)) = \cos \left(\frac{1}{t}\right) \cdot f(t)$$
for all $f \in \mathscr{A}$ and $t \in (0,1]$. Then ${\rm ran}(S) \not= {\rm ran}(T)$ and they are neither
equal to  $\mathscr{E}$, but ${\rm ran}(S)^{\bot\bot} = {\rm ran}(T)^{\bot\bot} =
\mathscr{E}$. Moreover, $\gamma = \sin \left(\frac{1}{t}\right) \cos \left(\frac{1}{t}\right) \in
M(\mathscr{A})$ is not invertible. Both the operators are bounded, adjointable (and hence
$\mathscr{A}$-linear), injective and orthogonality-preserving, but they are not invertible.
\end{example}
\begin{corollary}\label{Cor_3.2}
Let ${\mathscr A}$ be a $C^*$-algebra and ${\mathscr E}$ be a full
Hilbert ${\mathscr A}$-module. Let $T: {\mathscr E} \to {\mathscr E}$
be a bounded ${\mathscr A}$-linear operator such that $x \perp y$
implies $T(x) \perp y$ for every suitable elements $x, y \in {\mathscr E}$.
Then there exists an element $\gamma \in Z(M({\mathscr A}))$ such that
$T(x)= \gamma x$ for each $x \in {\mathscr E}$.
\end{corollary}

The proof is the same as for Corollary \ref{Cor_2.10} changing the origin of
$\gamma$ and the theorem referred to.
\begin{corollary}
Let $\mathscr A$ be a $C^*$-algebra and $\mathscr E$ be a full Hilbert $\mathscr A$-module.
Let $U,V : \mathscr E \to \mathscr E$ be isometries with $UU^*=VV^*=I$.
Then $x \perp y$ forces $U(x) \perp V(y)$ for every suitable pair $x,y \in \mathscr E$
if and only if either $U(\mathscr E) \perp V(\mathscr E)$ or $U(x)=\gamma V(x)$
for a fixed unitary element $\gamma \in Z(M({\mathscr A}))$ and every $x \in \mathscr E$.
\end{corollary}

\begin{proof}
Let $x,y \in \mathscr E$ be a pair of orthogonal elements.
Then $\langle V^*U(x),y \rangle = \langle U(x),V(y) \rangle = 0$ by the assumption.
Applying Corollary \ref{Cor_3.2} we obtain either $V^*U=0$ or $V^*U(x)=\gamma x$
for a certain non-zero $\gamma \in Z(M({\mathscr A}))$,
i.e. $U=\gamma V$, and $\gamma$ has to be unitary since we are treating co-isometries.
\end{proof}

Geometrically this means that two isometric copies of a Hilbert $\mathscr A$-module $\mathscr E$
embedded into $\mathscr E$ as orthogonal direct summands additionally preserve orthogonality of
the different images of each pair of initially orthogonal elements of $\mathscr E$, if and only
if either these two images of $\mathscr E$ are orthogonal to each other as $\mathscr A$-submodules
or these two isometric embeddings as orthogonal direct summands only differ by multiplication
by a unitary from $Z(M({\mathscr A}))$, i.e. coincide as $\mathscr A$-modules.

The next theorem generalizes \cite[Theorem 3]{F.M.P} and \cite[Theorem 2.3]{L.N.W.1},
and gives a partial solution of \cite[Problem 1]{F.M.P}. It extends results of \cite{Ch.new}.

\begin{theorem} \label{Thm_3.2}
Let ${\mathscr A}$ be a $C^*$-algebra and ${\mathscr E}$ be a full
Hilbert ${\mathscr A}$-module. Let $T: {\mathscr E} \to {\mathscr E}$
be a linear operator and let $S: {\mathscr E} \to {\mathscr E}$ be
an invertible linear operator with bounded inverse operator.
Suppose, $\langle T(x), S(y) \rangle = \gamma \langle x,y \rangle$
for some invertible element $\gamma \in Z(M({\mathscr A}))$.
Then $T$ and $S$ are bounded, ${\mathscr A}$-linear, adjointable, invertible,
and $ST^*(x)= \gamma x$ and $S^*T(x) = \gamma^* x$
for all $x \in {\mathscr E}$, i.e. the pairs of operators $\{S, T^* \}$
and $\{S^*, T \}$ commute, and also $S=\gamma^{-1} {(T^*)}^{-1}$
and $T = \gamma^*{(S^*)}^{-1}$.
\newline
In the special situation of $T=S$ the element $\gamma$ is positive
and $T=\sqrt{\gamma}U$ for some unitary ${\mathscr A}$-linear
operator $U$ on ${\mathscr E}$.
\end{theorem}

\begin{remark}
In the case where the Hilbert ${\mathscr A}$-module ${\mathscr E}$ admits some invertible
bounded ${\mathscr A}$-linear operator $S$ that is not adjointable,
it can not fulfil the equality $\langle T(x), S(y) \rangle = \gamma \langle x,y \rangle$ for any
bounded ${\mathscr A}$-linear operator $T$ on ${\mathscr E}$ and any $\gamma \in Z(M({\mathscr A}))$;
see  \cite[Example 6.2]{Brown} and \cite[Example 7.3]{FrankPositivity}.
\end{remark}

\begin{proof}
Since $S$ is boundedly invertible, we derive the equality $\gamma \langle x, S^{-1}(z)
\rangle = \langle T(x), z \rangle$ for every $x,z \in {\mathscr E}$. By
the boundedness of $S^{-1}$, the operator $T$ is adjointable, bounded and $T^*=\gamma S^{-1}$.
Since adjointable linear operators on Hilbert ${\mathscr A}$-modules are
${\mathscr A}$-linear and bounded, both $T$ and $S$ have to be
${\mathscr A}$-linear, bounded, invertible with bounded inverses and adjointable.
This proves the last assertion.
So $TT^{-1}={\gamma^*}^{\,-1}TS^*=I$ and
$T^{-1}T={\gamma^*}^{\,-1}S^*T=I$.
We arrive at $TS^* = \gamma^*I$ and $S^*T=\gamma^*I$.
We obtain the commutation result.

If, in particular, $T=S$ in our initial equality we derive $TT^*=\gamma I$ and
$T^*T = \gamma^*I$. Consequently, $\gamma$ is positive and
$T = \sqrt{\gamma} U$ for some ${\mathscr A}$-linear
unitary operator $U$ on ${\mathscr E}$. This shows the last assertion.
\end{proof}

\begin{remark}
Checking the conditions on suitable bijective operators $T$ on Hilbert $C^*$-modules
fulfilling the conditions of Theorem \ref{Thm_3.2} one recognizes that every bounded adjointable
invertible operator $T$ together with the operator $S=\gamma^{-1} (T^*)^{-1}$ satisfies the
equality $\langle T(x), S(y) \rangle = \gamma \langle x,y \rangle$ for all elements
$x$ and $y$. The adjointability of $T$ is necessary. In the particular case of $T$ being unitary
and $\gamma = 1$, the operator $S$ has to be unitary, too, and $T=S$. Moreover, for a
given operator $T$, the operator $S$ is unique.
Comparing these observations with the results in \cite{Ch.new} in the setting of Hilbert spaces,
the case of non-surjective operators $T$ and of non-injective elements $\gamma$ have to be
investigated in more details. The (norm-closure of the) range of $T$ might not be an orthogonal
summand, in particular. So more various situations will appear, see Section 6.
\end{remark}

\begin{corollary}
Let $\mathscr A$ be a $C^*$-algebra and let $\langle \cdot,\cdot \rangle_2$ be another
$\mathscr A$-valued inner product on a full Hilbert $\mathscr A$-module $\{ {\mathscr E},
{\langle \cdot,\cdot \rangle}_1 \}$ inducing an equivalent norm to the given one.
Suppose, ${\langle x,y \rangle}_1=0$ implies ${\langle x,y \rangle}_2=0$ for
every suitable $x,y \in {\mathscr E}$. Then there exists an invertible positive
element $\gamma \in Z(M({\mathscr A}))$ such that ${\langle x,y \rangle}_1 =
\gamma {\langle x,y \rangle}_2$ for all $x,y \in {\mathscr E}$.
\end{corollary}
\begin{proof}
Set ${\mathscr F}={\mathscr E}$ as an $\mathscr A$-module, and add
the alternative $\mathscr A$-valued inner product ${\langle \cdot,\cdot \rangle}_2$.
Set $T=S=I$ and note, that both of these operators are bounded if
considered as operators on $\mathscr E$. Then
Theorem \ref{Thm_3.1} and Corollary \ref{Cor_3.2} yield
${\langle x,y \rangle}_1 = \gamma {\langle x,y \rangle}_2$ for all
$x,y \in {\mathscr E}$.
\end{proof}

\section{Results in $C^*$-algebra of real rank zero}
Recall that a $C^{*}$-algebra $\mathscr{A}$ has real rank zero if every self-adjoint
element in $\mathscr{A}$ can be approximated in norm by invertible self-adjoint elements.
Note that if $\mathscr{A}$ has real rank zero, then the $*$-algebra generated by all
the idempotents in $\mathscr{A}$ is dense in $\mathscr{A}$; see, for example, \cite{B.P}.
The result extends \cite[Theorem 2.3]{L.N.W.1}.

\begin{theorem}\label{th.213.70}
Let $\mathscr{A}$ be a $C^*$-algebra of real rank zero and with identity $e$, and let $\mathscr{E}$
and $\mathscr{F}$ be Hilbert $\mathscr{A}$-modules. Suppose that $\mathscr{A}$-linear operators
$T, S:\mathscr{E}\to \mathscr{F}$ satisfy the condition
\begin{align*}
x\perp y\, \Longrightarrow \,T(x)\perp S(y) \qquad (x, y\in \mathscr{E}).
\end{align*}
Suppose that there is $z\in \mathscr{E}$ with $\langle z, z\rangle$ being invertible
and $\langle T(z), S(z)\rangle$ is self-adjoint. Then, there exits a self-adjoint element
$\gamma$ in the center $Z(\mathscr{A})$ of $\mathscr{A}$ such that
$$\langle T(x), S(y)\rangle = \gamma\langle x, y\rangle \qquad (x, y\in \mathscr{E}).$$
\end{theorem}

\begin{proof}
By replacing $z$ with $|z|^{-1}z$, we assume $\langle z, z\rangle = e$.
For every symmetry (i.e. a self-adjoint unitary) $u\in \mathscr{A}$, we have
\begin{align*}
\langle z + uz, z - uz\rangle = |z|^2 + u|z|^2 - u|z|^2 - u|z|^2u = 0
\end{align*}
whence, $z + uz \perp z - uz$. Hence our assumption yields $T(z + uz)\perp S(z - uz)$, or equivalently
\begin{align*}
\langle T(z), S(z)\rangle + u\langle T(z), S(z)\rangle - \langle T(z), S(z)\rangle u - u\langle T(z), S(z)\rangle u = 0.
\end{align*}
Now, let $\gamma : = \langle T(z), S(z)\rangle$. So, the above equality becomes $\gamma + u\gamma - \gamma u - u\gamma u = 0$.
Since $\gamma$ is self-adjoint, by taking adjoint
$\gamma + \gamma u - u\gamma - u\gamma u = 0$. Thus $u \gamma = \gamma u$.
As $\mathscr{A}$ is generated by projections, and thus also by symmetries, we get $\gamma\in Z(\mathscr{A})$.
On the other hand, for each $x\in \mathscr{E}$ we have
$\big\langle z, x - \langle x, z\rangle z\big\rangle = \langle z, x\rangle - |z|^2\langle z, x\rangle = 0$.
Hence
\begin{align}\label{id.212.788901}
z\perp x- \langle x, z\rangle z \qquad \mbox{and} \qquad x- \langle x, z\rangle z\perp z.
\end{align}
So, our assumption yields
\begin{align}\label{id.212.788902}
T(z)\perp S\big(x - \langle x, z\rangle z\big) \qquad \mbox{and} \qquad T\big(x - \langle x, z\rangle z\big) \perp S(z).
\end{align}
Furthermore, from (\ref{id.212.788901}) we infer that
\begingroup\makeatletter\def\f@size{10}\check@mathfonts
\begin{align*}
\Big\langle x - \langle x, z\rangle z &+ \big|x - \langle x, z\rangle z\big|z, x - \langle x, z\rangle z - \big|x - \langle x, z\rangle z\big|z\Big\rangle
\\& = \big|x - \langle x, z\rangle z\big|^2 - \big\langle x - \langle x, z\rangle z, z\big\rangle\big|x - \langle x, z\rangle z\big|
\\& \qquad \qquad \qquad \qquad + \big|x - \langle x, z\rangle z\big|\big\langle z, x - \langle x, z\rangle z\big\rangle - \big|x - \langle x, z\rangle z\big|\langle z, z\rangle\big|x - \langle x, z\rangle z\big|
\\& = \big|x - \langle x, z\rangle z\big|^2 - \big|x - \langle x, z\rangle z\big|^2 = 0.
\end{align*}
\endgroup
Then $x - \langle x, z\rangle z + \big|x - \langle x, z\rangle z\big|z$
is orthogonal to $x - \langle x, z\rangle z - \big|x - \langle x, z\rangle z\big|z$
and hence $T\Big(x - \langle x, z\rangle z + \big|x - \langle x, z\rangle z\big|z\Big)$ is orthogonal to
$S\Big(x - \langle x, z\rangle z - \big|x - \langle x, z\rangle z\big|z\Big)$.
Thus, it follows from (\ref{id.212.788902}) that
\begingroup\makeatletter\def\f@size{10}\check@mathfonts
\begin{align*}
0 & = \Big\langle T\Big(x - \langle x, z\rangle z + \big|x - \langle x, z\rangle z\big|z\Big), S\Big(x - \langle x, z\rangle z - \big|x - \langle x, z\rangle z\big|z\Big)\Big\rangle
\\& = \Big\langle T\big(x - \langle x, z\rangle z\big), S\big(x - \langle x, z\rangle z\big)\Big\rangle -
\Big\langle T\big(x - \langle x, z\rangle z\big), S(z)\Big\rangle\big|x - \langle x, z\rangle z\big|
\\& \qquad \qquad \qquad + \big|x - \langle x, z\rangle z\big|\Big\langle T(z), S\big(x - \langle x, z\rangle z\big)\Big\rangle
-\big|x - \langle x, z\rangle z\big|\langle T(z), S(z)\rangle \big|x - \langle x, z\rangle z\big|
\\& = \Big\langle T\big(x - \langle x, z\rangle z\big), S\big(x - \langle x, z\rangle z\big)\Big\rangle
-\big|x - \langle x, z\rangle z\big|\gamma\big|x - \langle x, z\rangle z\big|.
\end{align*}
\endgroup
Then $\Big\langle T\big(x - \langle x, z\rangle z\big), S\big(x - \langle x, z\rangle z\big)\Big\rangle
= \big|x - \langle x, z\rangle z\big|\gamma\big|x - \langle x, z\rangle z\big|$.
Since $\gamma\in Z(\mathscr{A})$, by (\ref{id.212.788901}), we obtain
\begin{align}\label{id.212.788903}
\Big\langle T\big(x - \langle x, z\rangle z\big), S\big(x - \langle x, z\rangle z\big)\Big\rangle &= \gamma\big|x - \langle x, z\rangle z\big|^2 \nonumber
\\ & = \gamma\Big\langle x - \langle x, z\rangle z, x\Big\rangle - \gamma \Big\langle x - \langle x, z\rangle z, \langle x, z\rangle z\Big\rangle \nonumber
\\ & = \gamma |x|^2 - \gamma |\langle x, z\rangle|^2.
\end{align}
From (\ref{id.212.788902}) and (\ref{id.212.788903}) we infer that
\begingroup\makeatletter\def\f@size{10}\check@mathfonts
\begin{align*}
\langle T(x), S(x)\rangle & = \Big\langle T\big(x - \langle x, z\rangle z\big) + \langle x, z\rangle T(z), S\big(x - \langle x, z\rangle z\big) + \langle x, z\rangle S(z)\Big\rangle
\\& = \Big\langle T\big(x - \langle x, z\rangle z\big), S\big(x - \langle x, z\rangle z\big)\Big\rangle
+ \langle x, z\rangle \Big\langle T(z), S\big(x - \langle x, z\rangle z\big)\Big\rangle
\\& \qquad \qquad \qquad \qquad + \Big\langle T\big(x - \langle x, z\rangle z\big), S(z)\Big\rangle \langle z, x\rangle
+ \langle x, z\rangle \langle T(z), S(z)\rangle \langle z, x\rangle
\\& = \gamma |x|^2 - \gamma |\langle x, z\rangle|^2 + \gamma|\langle x, z\rangle|^2 = \gamma |x|^2.
\end{align*}
\endgroup
Hence
\begin{align}\label{id.212.788904}
\langle T(x), S(x)\rangle = \gamma|x|^2 \qquad (x\in \mathscr{E}).
\end{align}
Now, by (\ref{id.212.788904}) and the polarization identity, we obtain
$$\langle T(x), S(y)\rangle = \gamma\langle x, y\rangle \qquad (x, y\in \mathscr{E}).$$
\end{proof}

\begin{remark}
Notice that orthogonality preserving functions may be very
nonlinear and discontinuous; see \cite[Example 2]{J.C.4} and \cite{Ch.new}.
Now, let $\mathscr{E}$ be a Hilbert $\mathbb{K}(\mathscr{H})$-module
and let $\mathscr{F}$ be a Hilbert $\mathscr{A}$-module.
Let $g, h:\mathscr{E}\to \mathscr{F}$ be additive functions such that
$$x\perp y\, \Longrightarrow \,g(x)\perp h(y) \qquad (x, y\in \mathscr{E}).$$
Suppose that function $f:\mathscr{E}\to \mathscr{A}$
defined by $f(x) : = \langle g(x), h(x)\rangle$ is continuous.
Fix $x, y\in \mathscr{E}$ such that $x\perp y$. Hence $\langle x, y\rangle = \langle y, x\rangle = 0$.
Therefore $\langle g(x), h(y)\rangle = \langle g(y), h(x)\rangle = 0$. Then
\begin{align*}
f(x + y) &= \langle g(x + y), h(x + y)\rangle
\\& = \langle g(x), h(x)\rangle + \langle g(x), h(y)\rangle + \langle g(y), h(x)\rangle + \langle g(y), h(y)\rangle
\\& = \langle g(x), h(x)\rangle + \langle g(y), h(y)\rangle = f(x) + f(y).
\end{align*}
Thus $f$ is orthogonally additive. It follows from \cite[Theorem 5.4 (ii)]{I.T.Y} that
there are a unique continuous additive function $A:\mathscr{E}\to \mathscr{A}$
and a unique operator $\Phi:\langle \mathscr{E}, \mathscr{E}\rangle\to \mathscr{A}$ such that
$$f(x) = A(x) + \Phi(\langle x, x\rangle) \qquad (x\in \mathscr{E}).$$
\end{remark}

\section{Additional comments}

Let us briefly discuss some obstacles in the theory of Hilbert $C^*$-modules which
prevent a straightforward generalization of Hilbert space results on the subject of the
present paper, cf. \cite{Ch.new}, \cite[Theorem 11]{Lu}, \cite[Theorem 4]{Lu.Wo.1}. First of all,
biorthogonally closed Hilbert $C^*$-submodules very often cannot divided out as
(any kind of) direct summand of the hosting Hilbert $C^*$-module.

\begin{example}
Let us take $\mathscr A = C([0,2\pi])$, regarded as a Hilbert $C^*$-module over
itself in the natural way. Consider two multiplication operators
\[
T(x)(t) = \cos(t) \cdot x(t) \quad {\rm for} \,\, t \in [0,\pi/2] \quad {\rm and} \quad T(x)(t) = 0 \quad {\rm for} \,\, t \in [\pi/2,2\pi]
\]
\[
S(x)(t) = 0 \quad {\rm for} \,\, t \in [0,3\pi/2] \quad {\rm and} \quad S(x)(t) = \cos(t) \cdot x(t) \quad {\rm for} \,\, t \in [3\pi/2, 2\pi].
\]
Then $\langle T(x), S(y) \rangle =  \langle x,y \rangle = 0$ for every pairwise orthogonal elements $x, y \in \mathscr A$.
However, the intersection of the kernels of $T$ and $S$ is neither a direct orthogonal nor a direct topological summand of $\mathscr A$, beside it is both norm-closed and biorthogonally complemented. And the operators $T$ and $S$ are bounded and self-adjoint, and hence $\mathscr A$-linear and even positive. In fact, both they lack polar decomposition in $\mathscr A$. Similarly, both the biorthogonally complemented images of $T$ and $S$ have analogous properties like the intersection of the two kernels. Finally, there does not exist any bounded invertible module operator on the intersection of the kernels of $T$ and $S$ such that it could be continuously extended to $\mathscr A$ in such a way that the orthogonal complement of the intersection of the two kernels is contained in its kernel. This situation is different from the Hilbert space situation. For an example involving unbounded module operators we refer to \cite[Example 2.1]{AO}.
\end{example}

A second obstacle to be considered is the existence of direct topological summands in certain Hilbert $C^*$-modules which are not direct orthogonal summands.
Here non-adjointable module operators come into play.

\begin{example}
Let $\mathscr A$ be a unital $C^*$-algebra with a non-trivial norm-closed two-sided ideal $\mathscr I$.
Consider the Hilbert $\mathscr A$-module $\mathscr E = \mathscr A \oplus \mathscr I$ and its Hilbert $\mathscr A$-submodules
$\mathscr {E}_1 = \{ (i,i): i\in \mathscr I\}$ and $\mathscr {E}_2 = \mathscr A \oplus \{ 0 \}$.
Their intersection is the set $\{ 0 \}$, $\mathscr {E}_1$ is not a direct orthogonal summand, however $\mathscr {E}_1$
coincides with its biorthogonal complement with respect to $\mathscr E$. But, $\mathscr {E}_1$ is a direct topological summand of $\mathscr E$,
since $(a,i)=(i,i)+ (a-i,0)$ is the unique decomposition of any element of $\mathscr E$ into an element of $\mathscr {E}_1$
and an element of $\mathscr {E}_2$. So, there exists a bounded $\mathscr A$-linear idempotent operator
$P: \mathscr E\to \mathscr {E}_1$ such that $P$ is non-adjointable.

Similarly, $\mathscr {E}_3 = \{ (i, -i) : i\in \mathscr I\} = {\mathscr E}^\bot_1$ and $\mathscr {E}_2$ are such a pair according to the decomposition
$(a,i) = (-i, i) + (a+i,0)$ for arbitrary elements of $\mathscr E$.
So, there is a bounded, $\mathscr A$-linear, non-adjointable, idempotent operator $Q:\mathscr E\to \mathscr {E}_3$.

Consequently, $\langle P(x),Q(y) \rangle = \langle x,y \rangle = 0$ for every pairwise orthogonal elements $x,y \in \mathscr E$.
And by definition, the pair of operators $P$ and $Q$ is orthogonality-preserving, but both these operators are non-adjointable.
And their ranges are norm-closed, biorthogonally complemented Hilbert $\mathscr A$-submodules that are not orthogonal summands,
but they are topological direct summands. Moreover, $P$ and $Q$ are idempotents. Also, $PQ\not=0$ and $QP\not=0$.
Such situations cannot appear for Hilbert spaces, cf. \cite[Theorem 11]{Lu}, \cite[Theorem 4]{Lu.Wo.1}.
\end{example}

\smallskip
{\bf Acknowledgments.} We would like to thank the referees for their careful
reading of the manuscript and for their useful comments. The second author (corresponding author)
was supported by a grant from Ferdowsi University of Mashhad No. 2/53798.
\bibliographystyle{amsplain}

\begin{thebibliography}{99}

\bibitem{Ak1969} C. A. Akemann,
\textit{The general Stone--Weierstrass problem},
J. Funct. Anal. \textbf{4} (1969), 277--294.

\bibitem{A.R} Lj. Aramba\v{s}i\'{c} and R. Raji\'{c},
\textit{A strong version of the Birkhoff--James orthogonality in Hilbert $C^*$-modules},
Ann. Funct. Anal. \textbf{5} (2014), no. 1, 109--120.

\bibitem{A.R.2} Lj. Aramba{\v{s}}i\'c and R. Raji\'c,
\textit{On three concepts of orthogonality in Hilbert C*-modules},
Linear Multilinear Algebra \textbf{63} (2015), no. 7, 1485--1500.

\bibitem{AO} M. B. Asadi and F. Olyaninezhad,
\textit{Orthogonality preserving pairs of operators on Hilbert $C_0(Z)$-modules},
Linear Multilinear Algebra, doi: 10.1080/03081087.2020.1825610.

\bibitem{B.G} D. Baki\'{c} and B. Gulja\v{s},
\textit{Hilbert $C^*$-modules over $C^*$-algebras of compact operators},
Acta Sci. Math. (Szeged) \textbf{68} (2002) 249--269.

\bibitem{B.P} L. G. Brown and G. K. Pedersen,
\textit{$C^*$-algebras of real rank zero},
J. Funct. Anal. \textbf{99} (1991), 131--149.

\bibitem{Brown} L. G. Brown,
\textit{Close hereditary $C^*$-subalgebras and the structure of quasi-multipliers},
MSRI preprint {\#} 11211-85, Purdue University, West Lafayette, USA, 1985;
Proceedings of the Royal Society of Edinburgh Section A: Mathematics \textbf{147}
(2017), no. 2, 263--292.

\bibitem{C} G. Chevalier,
\textit{Wigner's theorem and its generalizations},
in: Handbook of Quantum Logic and Quantum Structures, eds. K. Engesser, D. M. Gabbay, D. Lehmann,
Elsevier, 2007, 429--475.

\bibitem{J.C.4} J. Chmieli\'{n}ski,
\textit{Linear mappings approximately preserving orthogonality},
J. Math. Anal. Appl. \textbf{304} (2005), 158--169.

\bibitem{Ch.new} J. Chmieli\'{n}ski,
\textit{Orthogonality equation with two unknown functions},
Aequationes  Math. \textbf{90} (2016), 11--23.

\bibitem{C.L.W} J. Chmieli\'{n}ski, R. {\L}ukasik, and P. W\'{o}jcik,
\textit{On the stability of the orthogonality equation and the orthogonality-preserving property with two unknown functions},
Banach J. Math. Anal. \textbf{10} (2016), no. 4, 828--847.

\bibitem{Frank:ZAA} M. Frank,
\textit{Self-duality and $C^*$-reflexivity of Hilbert $C^*$-modules},
Z. Anal. Anwend. \textbf{9} (1990), 165--176.

\bibitem{FrankPositivity} M. Frank,
\textit{Geometrical aspects of Hilbert $C^*$-modules}.
Positivity \textbf{3} (1999), no. 3, 215--243.

\bibitem{F.M.P} M. Frank, A. S. Mishchenko and A. A. Pavlov,
\textit{Orthogonality-preserving, $C^*$-conformal and conformal module mappings on Hilbert $C^*$-modules},
J. Funct. Anal. \textbf{260} (2011), 327--339.

\bibitem{I.T} D. Ili\v{s}evi\'{c} and A. Turn\v{s}ek,
\textit{Approximately orthogonality preserving mappings on $C^*$-modules},
J. Math. Anal. Appl. \textbf{341} (2008), 298--308.

\bibitem{I.T.Y} D. Ili\v{s}evi\'{c}, A. Turn\v{s}ek and D. Yang,
\textit{Orthogonally additive mappings on Hilbert modules},
Studia Math. \textbf{221} (2014), 209--229.

\bibitem{L.N.W.1} C.-W. Leung, C.-K. Ng and N.-C. Wong,
\textit{Linear orthogonality preservers of Hilbert $C^*$-modules over $C^*$-algebras with real rank zero},
Proc. Amer. Math. Soc. \textbf{140} (2012), no. 9, 3151--3160.

\bibitem{L.N.W.2} C.-W. Leung, C.-K. Ng and N.-C. Wong,
\textit{Automatic continuity and $C_0(\Omega)$-linearity of linear maps between $C_0(\Omega)$-modules},
J. Operator Theory \textbf{67} (2012), no. 1, 3--20.

\bibitem{Li} H. Li,
\textit{A Hilbert $C^*$-module admitting no frames},
Bull. London Math. Soc. \textbf{42} (2010), 388--394.

\bibitem{LT} Y. Li and D. Tan,
\textit{Wigner's theorem on the Tsirelson space $T$},
Ann. Funct. Anal. \textbf{10} (2019), no. 4, 515--524.

\bibitem{Lu} R. {\L}ukasik,
\textit{A note on the orthogonality equation with two functions},
Aequationes  Math. \textbf{90} (2016), no. 5, 961--965.

\bibitem{Lu.Wo.1} R. {\L}ukasik and P. W\'ojcik,
\textit{Decomposition of two functions in the orthogonality equation},
Aequationes  Math. \textbf{90} (2016), no. 3, 495--499.

\bibitem{Lu.Wo.2} R. {\L}ukasik and P. W\'ojcik,
\textit{Functions preserving the biadditivity},
Results Math. \textbf{75} 82 (2020).

\bibitem{M.S.P} A. Mal, D. Sain and K. Paul,
\textit{On some geometric properties of operator spaces},
Banach J. Math. Anal. \textbf{13} (2019), no. 1, 174--191.

\bibitem{Man} V. M. Manuilov and E. V. Troitsky,
\textit{Hilbert $C^*$-modules},
In: Translations of Mathematical Monographs. \textbf{226},
American Mathematical Society, Providence, RI, 2005.

\bibitem{MOL} L. Moln\'ar,
\textit{Orthogonality preserving transformations on indefinite inner product spaces: generalization of Uhlhorn's version of Wigner's theorem},
J. Funct. Anal. \textbf{194} (2002), no. 2, 248--262.

\bibitem{Mol2000} L. Moln\'ar,
\textit{Generalizations of Wigner's unitary-antiunitary theorem for indefinite inner product spaces},
Comm. Math. Phys. \textbf{201} (2000), 785--791.

\bibitem{M.Z} M. S. Moslehian and A. Zamani,
\textit{Mappings preserving approximate orthogonality in Hilbert $C^*$-modules},
Math Scand. \textbf{122} (2018), 257--276.

\bibitem{P.S.M.M} K. Paul, D. Sain, A. Mal and K. Mandal,
\textit{Orthogonality of bounded linear operators on complex Banach spaces},
Adv. Oper. Theory \textbf{3}, (2018), no. 3, 699--709.

\bibitem{Pa1973} W. L. Paschke,
\textit{Inner product modules over $B^*$-algebras},
Trans. Amer. Math. Soc. \textbf{182} (1973), 443--468.

\bibitem{Ped1979} G. K. Pedersen,
\textit{$C^*$-algebras and their automorphism groups},
Academic Press, London--New York--San Francisco, 1979.

\bibitem{RS} L. Rodman and P. \v{S}emrl,
\textit{Orthogonality preserving bijective maps on finite dimensional projective spaces over division rings},
Linear Multilinear Algebra \textbf{56} (2008), no. 6, 647--664.

\bibitem{M.M.S} M. M. Sadr,
\textit{Decomposition of functions between Banach spaces in the orthogonality equation},
Aequationes  Math. \textbf{91} (2017), no. 4, 739--743.

\bibitem{Takesaki} M. Takesaki,
\textit{Theory of Operator Algebras I},
Encyclopedia Math. Sciences v. 124, Springer, 1979 / 2002.

\bibitem{A.T.1} A. Turn\v{s}ek,
\textit{On mappings approximately preserving orthogonality},
J. Math. Anal. Appl. \textbf{336} (1) (2007), 625--631.

\bibitem{U} U. Uhlhorn,
\textit{Representation of symmetry transformations in quantum mechanics},
Ark. Fysik, \textbf{23} (1962), 307--340.

\bibitem{W} E. Wigner,
\textit{Gruppentheorie und ihre Anwendung auf die Quantenmechanik der Atomspektren},
Vieweg, Braunschweig, 1931.

\bibitem{P.W} P. W\'{o}jcik,
\textit{On certain basis connected with operator and its applications},
J. Math. Anal. Appl. \textbf{423} (2) (2015), 1320--1329.

\bibitem{W0} P. W\'ojcik,
\textit{On a functional equation characterizing linear similarities},
Aequationes  Math. \textbf{93} (2019), no. 3, 557--561.

\bibitem{Z.M.F} A. Zamani, M. S. Moslehian and M. Frank,
\textit{Angle preserving mappings},
Z. Anal. Anwend. \textbf{34} (2015), 485--500.

\end{thebibliography}

\end{document}